\documentclass[a4paper]{amsart}
\usepackage{amsmath,amsthm,amssymb,latexsym,epic,bbm,comment,color}
\usepackage{graphicx,enumerate,stmaryrd}
\usepackage[all,2cell]{xy}
\xyoption{2cell}

\newtheorem{theorem}{Theorem}

\newtheorem{corollary}[theorem]{Corollary}
\newtheorem{proposition}[theorem]{Proposition}

\newtheorem{remark}[theorem]{Remark}
\usepackage[all]{xy}
\usepackage[active]{srcltx}
\usepackage[parfill]{parskip}

\font\sc=rsfs10
\newcommand{\cC}{\sc\mbox{C}\hspace{1.0pt}}

\newcommand{\cD}{\sc\mbox{D}\hspace{1.0pt}}

\newcommand{\cA}{\sc\mbox{A}\hspace{1.0pt}}

\font\scc=rsfs7

\newcommand{\ccD}{\scc\mbox{D}\hspace{1.0pt}}

\begin{document}
\title[Pyramids and $2$-representations]
{Pyramids and $2$-representations}

\author{Volodymyr Mazorchuk, Vanessa Miemietz and Xiaoting Zhang}

\begin{abstract}
We describe a diagrammatic procedure which lifts strict monoidal 
actions from  additive categories to categories of complexes 
avoiding any use of direct sums. As an application, we prove that 
every simple transitive $2$-representation of the $2$-category of 
projective  bimodules over a finite dimensional algebra is equivalent 
to a cell $2$-representation.
\end{abstract}

\maketitle

\section{Introduction and description of the results}\label{s1}

One of the most fundamental results in the classical representation theory is
the fact that the algebra $\mathrm{Mat}_{n\times n}(\Bbbk)$ of $n\times n$ matrices
over an algebraically closed field $\Bbbk$ has only one isomorphism class of simple
modules. The main result of the present paper is an analogue of this latter 
fact in $2$-representation theory.

Modern  $2$-representation theory originates from \cite{BFK,CR,KL,Ro} which 
emphasize the use of $2$-categories and $2$-representations for solving 
various problems in algebra and topology. The series \cite{MM1,MM2,MM3,MM4,MM5,MM6}
of papers started a systematic study of $2$-representations of so-called 
{\em finitary} $2$-categories which are natural $2$-analogues of finite dimensional algebras. 
The weak Jordan-H{\"o}lder theory developed in \cite{MM5} motivates the study of
so-called {\em simple transitive} $2$-representations which are suitable $2$-analogues 
of simple modules. It turns out that, in many cases, simple transitive $2$-representations
can be explicitly classified, see \cite{MM5,MM6,MaMa,KMMZ,MMMT,MMZ,MT,MZ,MZ2,Zh,Zi1,Zi3} 
for the results and \cite{Ma} for a detailed survey on the subject. In many, but not all, cases,
simple transitive $2$-representations are exhausted by so-called {\em cell} $2$-representations
defined already in \cite{MM1}. Cell $2$-representations are precisely the weak 
Jordan-H{\"o}lder subquotients of regular (a.k.a. principal) $2$-representations. 

For the moment, a classification of {\em simple} finitary $2$-categories is only available under
substantial additional assumptions (in particular, that of existence of a weak anti-equivalence and
adjunction morphisms), see \cite{MM3}. Roughly speaking, in that special case simple 
$2$-categories are classified by $2$-categories of projective bimodules over finite dimensional
self-injective algebras. Here self-injectivity is crucial as it is equivalent to existence
of a weak anti-equivalence and adjunction morphisms. For a finite dimensional algebra $A$, the 
corresponding $2$-category of projective bimodules is usually denoted by $\cC_A$.
With this result in hand, it is very natural to
consider $2$-categories of projective bimodules for arbitrary finite  dimensional algebras as
a basic family of simple $2$-categories. Outside the self-injective case, a number of examples
were studied in \cite{MZ,MZ2,MMZ,Zi3} and in all cases, using rather different arguments, it was
shown that simple transitive $2$-representations are exhausted by cell $2$-representations.

The main result of the present paper, Theorem~\ref{thm81}, asserts that the latter is the case
for any finite dimensional algebra over an algebraically closed field. This answers (positively)
\cite[Question~15]{Ma}. Compared to all previous studies, our approach uses two crucial new ideas.

The first idea is related to creating some adjunction morphisms. In case $A$ has a non-zero
projective injective module, some of the non-identity $1$-morphisms in $\cC_A$ form adjoint pairs
of functors. This was exploited in \cite{MZ,MZ2,Zi3}. In the present paper we suggest to enlarge
$\cC_A$ to a $2$-category $\cD_A$ by adding right adjoints to all $1$-morphisms in $\cC_A$.
The $2$-category $\cD_A$ can be realized using $A$-$A$-bimodules by adding the 
$A$-$A$-bimodule $A^*\otimes_{\Bbbk}A$ as the right adjoint of the $A$-$A$-bimodule $A\otimes_{\Bbbk}A$.
The $2$-category $\cD_A$ loses some of the symmetries of $\cC_A$ but compensates for this loss by 
possessing many pairs of adjoint $1$-morphisms. With the machinery developed in 
\cite{KM2,KMMZ,MZ,MMZ} and some tricks using matrix computations, 
we prove in Theorem~\ref{thm71} that simple transitive $2$-representations of $\cD_A$
are exhausted by  cell $2$-representations.

The second idea is related to the necessity of some kind of ``induction'' allowing us to connect
$2$-representations of $\cC_A$ with $2$-representations of $\cD_A$. In classical representation theory,
induction is done using tensor products. Unfortunately, this technology is not yet available in
$2$-representation theory, which creates major obstacles. In the particular case of the $2$-categories
$\cC_A$ and $\cD_A$, we propose a way around the problem. We observe that $1$-morphisms in 
$\cD_A$ can be identified with (homotopy classes of) {\em complexes} of $1$-morphisms in $\cC_A$.
This raises the problem of lifting the strict $2$-structure from $\cC_A$ to the corresponding
homotopy category of complexes. The main obstacle is the incompatibility of strictness and additivity
of $1$-morphisms. To resolve this, we substitute the category of complexes by a new category, 
which we call the category of {\em pyramids}, see Section~\ref{s2}. This category is equivalent to
the category of complexes, but its tensor structure can be defined avoiding direct
sums (that is, avoiding the construction of taking the total complex). We use pyramids to lift 
$2$-representations of $\cC_A$ to the homotopy category of pyramids over $\cC_A$ which can then be
restricted to $\cD_A$ as the latter lives inside pyramids over $\cC_A$. This gives a well-defined
``induction'' from $\cC_A$ to $\cD_A$ which allows us to prove Theorem~\ref{thm81} using Theorem~\ref{thm71}.

The paper is organized as follows: in Section~\ref{s2} we collect all the results related to the
definition and properties of pyramids. Section~\ref{s3} collects preliminaries on 
$2$-categories and $2$-representations. In Section~\ref{s7} we study $2$-representations of 
$\cD_A$. The main result of this section is  Theorem~\ref{thm71}. In Section~\ref{s8} we formulate
and prove Theorem~\ref{thm81} and also give a characterization of $2$-categories of the
form $\cC_A$ inside the class of finitary $2$-categories, see Theorem~\ref{thm86}.

\textbf{Acknowledgements:} This research was partially supported by
the Swedish Research Council and G{\"o}ran Gustafsson Stiftelse. The major part of this
research was done during the visit, in April 2017, of the second author to Uppsala University,
whose hospitality is gratefully acknowledged.

\section{Pyramids}\label{s2}

\subsection{Indices}\label{s2.1}
We denote by  $\mathbb{N}$ the set of positive integers and
by $\mathbb{Z}_{\geq 0}$ the set of non-negative integers. Further, we denote by $\mathbb{I}$
the set of all vectors $\mathbf{a}=(a_i)_{i\in\mathbb{N}}$, written 
$\mathbf{a}=(a_1,a_2,a_3,\dots)$, where $a_i\in\mathbb{Z}$ and $a_i=0$ for $i\gg 0$.
Note that $\mathbb{I}$ is an abelian group with respect to component-wise addition. 
The zero element in $\mathbb{I}$ is $\mathbf{0}:=(0,0,\dots)$.

For $i\in\mathbb{N}$, we denote by $\varepsilon_i$ the element $(a_i)_{i\in\mathbb{N}}\in\mathbb{I}$
such that $a_j=\delta_{i,j}$ for all $j\in \mathbb{N}$ (here $\delta_{i,j}$ is the Kronecker symbol).
Then $\mathbb{I}$ is a free abelian group with basis $\mathbf{B}:=\{\varepsilon_i:i\in\mathbb{N}\}$, in particular,
each element of $\mathbb{I}$ can be written uniquely as a linear combination (over $\mathbb{Z}$)
of elements in $\mathbf{B}$. 

For $\mathbf{a}\in \mathbb{I}$, the {\em height} of $\mathbf{a}$ is defined to be
$\mathrm{ht}(\mathbf{a})=\sum_i a_i\in\mathbb{Z}$. Note that the latter is well-defined as
only finitely many components of $\mathbf{a}$ are non-zero. For $k\in\mathbb{Z}$, we 
denote by $\mathbb{I}_k$ the set of all $\mathbf{a}\in \mathbb{I}$ of height $k$.

Denote by $\pi_0:\mathbb{I}\to \mathbb{I}$ the map which maps all $\mathbf{a}$ to $\mathbf{0}$.
Let $\sigma_0:\mathbb{I}\to \mathbb{I}$ be the identity map. For $k\in\mathbb{N}$, define
$\pi_k:\mathbb{I}\to \mathbb{I}$ as the map sending $\mathbf{a}\in \mathbb{I}$ to $(a_1,a_2,\dots,a_k,0,0,\dots)$.
Define also $\sigma_k:\mathbb{I}\to \mathbb{I}$ as the map sending 
$\mathbf{a}\in \mathbb{I}$ to $(a_{k+1},a_{k+2},\dots)$.

\subsection{Pyramids over an additive category}\label{s2.2}

Let $\mathcal{A}$ be an additive category. A {\em pyramid} $(X_{\bullet},d_{\bullet},n)$ 
over $\mathcal{A}$ is a tuple
\begin{displaymath}
(X_{\bullet}:=\{X_{\mathbf{a}};\mathbf{a}\in \mathbb{I}\},
d_{\bullet}:=\{d_{\mathbf{a},i}:\mathbf{a}\in \mathbb{I},i\in\mathbb{N}\},n),
\end{displaymath}
where
\begin{itemize}
\item $n\in \mathbb{Z}_{\geq 0}$,
\item all $X_{\mathbf{a}}$ are objects in $\mathcal{A}$,
\item each $d_{\mathbf{a},i}$ is a morphism in $\mathcal{A}$ from 
$X_{\mathbf{a}}$ to $X_{\mathbf{a}+\varepsilon_i}$, 
\end{itemize}
satisfying the following axioms:
\begin{enumerate}[(I)]
\item\label{axiom1} we have $X_{\mathbf{a}}=0$ unless  $a_i=0$, for all $i>n$,
\item\label{axiom2} there is $m\in\mathbb{Z}$ such that $X_{\mathbf{a}}=0$, unless 
all $a_i<m$,
\item\label{axiom3} we have $d_{\mathbf{a}+\varepsilon_i,i}\circ d_{\mathbf{a},i}=0$, for all $\mathbf{a}$ and $i$,
\item\label{axiom4} we have $d_{\mathbf{a}+\varepsilon_i,j}\circ d_{\mathbf{a},i}=-
d_{\mathbf{a}+\varepsilon_j,i}\circ d_{\mathbf{a},j}$, for all $\mathbf{a}$, $i$ and $j$.
\end{enumerate}

For a pyramid $(X_{\bullet},d_{\bullet},n)$ over $A$ and $k\in\mathbb{Z}$, we define $d^{(k)}$ as the
matrix $(d^{(k)}_{\mathbf{a},\mathbf{b}})_{\mathbf{a}\in \mathbb{I}_{k+1}}^{\mathbf{b}\in \mathbb{I}_k}$,
where
\begin{displaymath}
d^{(k)}_{\mathbf{a},\mathbf{b}}=
\begin{cases}
d_{\mathbf{b},i}, &\text{ if }\mathbf{a}=\mathbf{b}+\varepsilon_{i};\\
0,&\text{ otherwise.} 
\end{cases}
\end{displaymath}

Let $(X_{\bullet},d_{\bullet},n)$ and $(Y_{\bullet},\partial_{\bullet},m)$ be two pyramids over $\mathcal{A}$.
A {\em morphism} 
\begin{displaymath}
\alpha:(X_{\bullet},d_{\bullet},n)\to (Y_{\bullet},\partial_{\bullet},m) 
\end{displaymath}
of pyramids is defined as $\alpha=\{\alpha^{(k)}:k\in\mathbb{Z}\}$, where each $\alpha^{(k)}$ is a matrix 
$(\alpha^{(k)}_{\mathbf{a},\mathbf{b}})_{\mathbf{a}\in \mathbb{I}_k}^{\mathbf{b}\in \mathbb{I}_k}$ with
$\alpha^{(k)}_{\mathbf{a},\mathbf{b}}:X_{\mathbf{b}}\to Y_{\mathbf{a}}$, such that the following condition
is satisfied, for every $k$:
\begin{equation}\label{eq1}
\alpha^{(k+1)} \cdot d^{(k)}=\partial^{(k)}\cdot \alpha^{(k)}.
\end{equation}
Here both sides of the equality should be understood as products of the corresponding matrices.
This is well-defined as, for each $k$,
the matrix $\alpha^{(k)}$ contains only finitely many non-zero components.

Let $(X_{\bullet},d_{\bullet},n)$, $(Y_{\bullet},\partial_{\bullet},m)$ and 
$(Z_{\bullet},\daleth_{\bullet},l)$ be three pyramids over $\mathcal{A}$. Let further 
\begin{displaymath}
\alpha:(X_{\bullet},d_{\bullet},n)\to (Y_{\bullet},\partial_{\bullet}.m)\text{ and }
\beta:(Y_{\bullet},d_{\bullet},m)\to (Z_{\bullet},\daleth_{\bullet},l) 
\end{displaymath}
be morphisms of pyramids. Then their {\em composition} 
$\beta\circ \alpha: (X_{\bullet},d_{\bullet},n)\to(Z_{\bullet},\daleth_{\bullet},l)$ is defined 
as the morphism $\gamma$ of pyramids such that $\gamma^{(k)}=\beta^{(k)}\cdot \alpha^{(k)}$, for each $k$.
Again, the right hand side should be understood as the usual product of matrices. Thanks to the finiteness
properties mentioned in the previous paragraph, the product is well-defined. Condition~\eqref{eq1} is
satisfied because of the computation
\begin{displaymath}
\beta^{(k+1)}\cdot \alpha^{(k+1)}\cdot d^{(k)}= \beta^{(k+1)}\cdot\partial^{(k)}\cdot \alpha^{(k)}=
\daleth^{(k)}\cdot \beta^{(k)}\cdot \alpha^{(k)},
\end{displaymath}
where the first equality is justified by the fact that $\alpha$ is a morphism of pyramids and the
second equality is justified by the fact that $\beta$ is a morphism of pyramids.

For a pyramid $(X_{\bullet},d_{\bullet},n)$, the corresponding identity morphism $\omega$ is defined
by declaring each $\omega^{(k)}$ to be the matrix such that
\begin{displaymath}
\omega^{(k)}_{\mathbf{a},\mathbf{b}}=
\begin{cases}
\mathrm{id}_{X_{\mathbf{a}}}, & \mathbf{a}=\mathbf{b},\\
0,& \text{otherwise.} 
\end{cases}
\end{displaymath}

\begin{proposition}\label{prop1}
Let $\mathcal{A}$ be an additive category.
The construct consisting of all pyramids over $\mathcal{A}$, morphisms of pyramids, 
composition of morphisms and identity morphisms forms a category, denoted $\mathcal{P}(\mathcal{A})$.
\end{proposition}

\begin{proof}
This follows directly from the definitions using the interpretation via matrix multiplication. 
\end{proof}

The category $\mathcal{P}(\mathcal{A})$ inherits from $\mathcal{A}$ the obvious preadditive structure given 
by com\-po\-nent-wise addition of morphisms. Furthermore, $\mathcal{P}(\mathcal{A})$ also inherits from $\mathcal{A}$
the additive structure by taking component-wise direct sums.

We have the canonical embedding of $\mathcal{A}$ into $\mathcal{P}(\mathcal{A})$ which sends an object 
$X\in \mathcal{A}$ to a pyramid concentrated at position $\mathbf{0}$ with the obvious assignment on
morphisms.

\subsection{Homotopy category of pyramids}\label{s2.3}

Let $(X_{\bullet},d_{\bullet},n)$ and $(Y_{\bullet},\partial_{\bullet},m)$ be two pyramids over $\mathcal{A}$.
A {\em homotopy} 
\begin{displaymath}
\chi:(X_{\bullet},d_{\bullet},n)\to (Y_{\bullet},\partial_{\bullet},m) 
\end{displaymath}
of pyramids is defined as $\chi=\{\chi^{(k)}:k\in\mathbb{Z}\}$, where each $\chi^{(k)}$ is a matrix 
$(\chi^{(k)}_{\mathbf{a},\mathbf{b}})_{\mathbf{a}\in \mathbb{I}_{k-1}}^{\mathbf{b}\in \mathbb{I}_k}$ with
$\chi^{(k)}_{\mathbf{a},\mathbf{b}}:X_{\mathbf{b}}\to Y_{\mathbf{a}}$.

If $\alpha:(X_{\bullet},d_{\bullet},n)\to (Y_{\bullet},\partial_{\bullet},m)$ is a morphism of pyramids,
we will say that the morphism $\alpha$ is {\em homotopic to zero}, denoted $\alpha\sim 0$, provided that there exists 
a homotopy $\chi:(X_{\bullet},d_{\bullet},n)\to (Y_{\bullet},\partial_{\bullet},m)$ such that 
\begin{displaymath}
\alpha^{(k)}=\chi^{(k+1)}\circ  d^{(k)}+ \partial^{(k-1)}\cdot \chi^{(k)}.
\end{displaymath}
As usual, null homotopic morphisms form an ideal of $\mathcal{P}(\mathcal{A})$, denoted $\mathcal{I}$,
and hence we may form the quotient $\mathcal{H}(\mathcal{A}):=\mathcal{P}(\mathcal{A})/\mathcal{I}$, which we will
call the {\em homotopy category of pyramids}.

\subsection{Pyramids versus complexes}\label{s2.4}
Let $\mathcal{A}$ be an additive category.
We denote by $\mathrm{Com}^-(\mathcal{A})$ the category of {\em bounded from the right} complexes over 
$\mathcal{A}$ and by $\mathcal{K}^-(\mathcal{A})$ the corresponding homotopy category.

The category $\mathrm{Com}^-(\mathcal{A})$ can be regarded as a subcategory of 
$ \mathcal{P}(\mathcal{A})$ in the obvious way, that is, by identifying the complex
\begin{equation}\label{eq2}
\dots \longrightarrow M_{k-1}\overset{f_{k-1}}{\longrightarrow}
M_{k}\overset{f_{k}}{\longrightarrow} M_{k+1}\longrightarrow\dots
\end{equation}
with the pyramid $(X_{\bullet},d_{\bullet},1)$, where 
\begin{displaymath}
X_{\mathbf{a}}=
\begin{cases}
M_k,&  \mathbf{a}=k\varepsilon_1,\\
0, & \text{ otherwise}; 
\end{cases}
\qquad\text{ and }
\qquad
d_{\mathbf{a},i}=
\begin{cases}
f_k,&  \mathbf{a}=k\varepsilon_1\text{ and }i=1,\\
0, & \text{ otherwise}. 
\end{cases}
\end{displaymath}

We can also define a functor $\mathcal{F}:\mathcal{P}(\mathcal{A})\to \mathrm{Com}^-(\mathcal{A})$ by sending
a pyramid $(X_{\bullet},d_{\bullet},n)$ to the complex of the form \eqref{eq2}  where 
\begin{displaymath}
M_k:=\bigoplus_{\mathrm{ht}(\mathbf{a})=k}X_{\mathbf{a}} 
\qquad\text{ and }\qquad
f_k=d^{(k)},
\end{displaymath}
with the action of $d^{(k)}$ on $M_k$ being the obvious one. Conditions \eqref{axiom1} and 
\eqref{axiom2} guarantee that $M_k$ is well-defined while conditions \eqref{axiom3} and 
\eqref{axiom4} imply that \eqref{eq2} is indeed a complex. On morphisms in $\mathcal{P}(\mathcal{A})$ the 
functor $\mathcal{F}$ is defined in the obvious way (using the natural action of a matrix
on a direct sum whose components index the columns of the matrix).

\begin{theorem}\label{thm2}
The functor $\mathcal{F}$ and the inclusion  of $\mathrm{Com}^-(\mathcal{A})$  into 
$\mathcal{P}(\mathcal{A})$ form a pair of  mutually quasi-inverse equivalences of categories.
\end{theorem}

\begin{proof}
Including and then applying  $\mathcal{F}$ does nothing and hence is obviously isomorphic to the 
identity functor. On the other hand, given a pyramid $(X_{\bullet},d_{\bullet},n)$, we can define 
a morphism from $(X_{\bullet},d_{\bullet},n)$ to  $\mathcal{F}(X_{\bullet},d_{\bullet},n)$ using
the obvious inclusion of each $X_{\mathbf{a}}$ into the corresponding $M_k$. This gives a natural
transformation from the identity functor to $\mathcal{F}$ followed by the 
inclusion of $\mathrm{Com}^-(\mathcal{A})$  into $\mathcal{P}(\mathcal{A})$. 
Projecting $M_k$ onto every $X_{\mathbf{a}}$ defines an inverse natural transformation. 
Therefore  applying  $\mathcal{F}$ and then including is also isomorphic to the identity 
functor. The claim follows.
\end{proof}

The following is now clear by comparing the definitions.

\begin{corollary}\label{cor3}
The mutually inverse equivalences in Theorem~\ref{thm2} induce mutually inverse equivalences between
$\mathcal{K}^-(\mathcal{A})$ and $\mathcal{H}(\mathcal{A})$.
\end{corollary}

\subsection{Tensoring pyramids}\label{s2.5}
This subsection will hopefully clarify why we need pyramids. Let $\mathcal{A}$ be an additive 
strict monoidal category. We denote the tensor product in $\mathcal{A}$ by $\circ$ and 
the identity object in $\mathcal{A}$ by $\mathbbm{1}$. We assume that  $\circ$ is biadditive.
We think of $\mathcal{A}$ as a $2$-category with one object and denote by $\circ_0$ the tensor 
product of morphisms and by $\circ_1$  the usual composition of morphisms in $\mathcal{A}$.
We would like to extend the monoidal structure on $\mathcal{A}$ to $\mathrm{Com}^-(\mathcal{A})$ and
to $\mathcal{K}^-(\mathcal{A})$. However, we do not know how to do that. The problem is that
to make this work one has to use the construction of taking the total complex, which involves taking
direct sums. However, there is usually no {\em strict} distributivity in $\mathcal{A}$ and hence
it is not possible to ensure strict associativity of the product of complexes. Our idea is to 
substitute the category of complexes by the category of pyramids where the tensor structure can
be extended without taking any direct sums. Here it will also become clear how the last component
of the pyramid tuple is used. The following construction is inspired by and generalizes \cite[Section~3]{MMMT}.

For two pyramids $(X_{\bullet},d_{\bullet},n)$ and $(Y_{\bullet},\partial_{\bullet},m)$ we define
their tensor product
\begin{displaymath}
(X_{\bullet},d_{\bullet},n)\circ (Y_{\bullet},\partial_{\bullet},m)
\end{displaymath}
as the pyramid $(Z_{\bullet},\daleth_{\bullet},n+m)$, where, for $\mathbf{a}\in\mathbb{I}$, we have
\begin{displaymath}
Z_{\mathbf{a}}:=X_{\pi_n(\mathbf{a})}\circ Y_{\sigma_n(\mathbf{a})},
\end{displaymath}
and, for $\mathbf{a}\in\mathbb{I}$ and $i\in \mathbb{N}$, we define
\begin{displaymath}
\daleth_{\mathbf{a},i}:=
\begin{cases}
\hspace{15mm}d_{\pi_n(\mathbf{a}),i}\circ_0\mathrm{id}, & \text{ if } i\leq n,\\
(-1)^{\mathrm{ht}(\pi_n(\mathbf{a}))}\,\,\,\,\,\,\mathrm{id}\circ_0\partial_{\sigma_n(\mathbf{a}),i}& \text{ otherwise}.
\end{cases}
\end{displaymath}
Let $\alpha:(X_{\bullet},d_{\bullet},n)\to(\tilde{X}_{\bullet},\tilde{d}_{\bullet},\tilde{n})$
and $\beta:(Y_{\bullet},\partial_{\bullet},m)\to(\tilde{Y}_{\bullet},\tilde{\partial}_{\bullet},\tilde{m})$
be morphisms of pyramids. Their tensor product 
\begin{displaymath}
\alpha\circ_0\beta:(X_{\bullet},d_{\bullet},n)\circ (Y_{\bullet},\partial_{\bullet},m)\to
(\tilde{X}_{\bullet},\tilde{d}_{\bullet},\tilde{n})\circ (\tilde{Y}_{\bullet},\tilde{\partial}_{\bullet},\tilde{m})
\end{displaymath}
is defined by
\begin{displaymath}
(\alpha\circ_0\beta)_{\mathbf{a},\mathbf{b}}^{(k)}:=
\begin{cases}
\alpha^{(l)}_{\pi_n(\mathbf{a}),\pi_{\tilde{n}}(\mathbf{b})}\circ_0 
\beta^{(k-l)}_{\sigma_n(\mathbf{a}),\sigma_{\tilde{n}}(\mathbf{b})},&
\text{ if }l=\mathrm{ht}(\pi_n(\mathbf{a}))=\mathrm{ht}(\pi_{\tilde{n}}(\mathbf{b})),\\
0,& \text{ otherwise},
\end{cases}
\end{displaymath}
for any $k\in\mathbb{Z}$ and any $\mathbf{a},\mathbf{b}\in\mathbb{I}_k$. Note that, under the assumption
$\mathbf{a},\mathbf{b}\in\mathbb{I}_k$, the conditions  
$l=\mathrm{ht}(\pi_n(\mathbf{a}))=\mathrm{ht}(\pi_{\tilde{n}}(\mathbf{b}))$
and $k-l=\mathrm{ht}(\sigma_n(\mathbf{a}))=\mathrm{ht}(\sigma_{\tilde{n}}(\mathbf{b}))$ are equivalent.

\begin{proposition}\label{prop4}
The above endows  $\mathcal{P}(\mathcal{A})$ with the structure of a strict monoidal category.
\end{proposition}

\begin{proof}
We start by checking that $(Z_{\bullet},\daleth_{\bullet},n+m)$ is indeed a pyramid.
It follows directly from the definitions
that  \eqref{axiom1}, \eqref{axiom2} and \eqref{axiom3} are satisfied. 
So, we only need to check \eqref{axiom4}. Let $i,j\in \mathbb{N}$
be different.
If both $i,j\leq n$, then  the corresponding part of \eqref{axiom4} 
for $(Z_{\bullet},\daleth_{\bullet},n+m)$ follows directly
from the definitions and \eqref{axiom4} for $(X_{\bullet},d_{\bullet},n)$.

Assume that both $i,j>n$. Then the anti-commutative square
\begin{displaymath}
\xymatrix{
Y_{\mathbf{c}+\varepsilon_i}\ar[r]^{\partial_{\mathbf{c}+\varepsilon_i,j}}&Y_{\mathbf{c}+\varepsilon_j\varepsilon_i} \\
Y_{\mathbf{c}}\ar[r]_{\partial_{\mathbf{c},j}}\ar[u]^{\partial_{\mathbf{c},i}}&
Y_{\mathbf{c}+\varepsilon_j}\ar[u]_{\partial_{\mathbf{c}+\varepsilon_j,i}}
}
\end{displaymath}
given by  \eqref{axiom4} for $(Y_{\bullet},\partial_{\bullet},m)$ induces one of the following
squares:
\begin{displaymath}
\xymatrix{
X_{\mathbf{b}}\otimes 
Y_{\mathbf{c}+\varepsilon_i}\ar[rr]^{-\mathrm{id}\otimes \partial_{\mathbf{c}+\varepsilon_i,j}}&&
X_{\mathbf{b}}\otimes Y_{\mathbf{c}+\varepsilon_j+\varepsilon_i} \\
X_{\mathbf{b}}\otimes Y_{\mathbf{c}}
\ar[rr]_{-\mathrm{id}\otimes \partial_{\mathbf{c},j}}\ar[u]^{-\mathrm{id}\otimes \partial_{\mathbf{c},i}}&&
X_{\mathbf{b}}\otimes Y_{\mathbf{c}+\varepsilon_j}\ar[u]_{-\mathrm{id}\otimes \partial_{\mathbf{c}+\varepsilon_j,i}}
}
\end{displaymath}
or
\begin{displaymath}
\xymatrix{
X_{\mathbf{b}}\otimes 
Y_{\mathbf{c}+\varepsilon_i}\ar[rr]^{\mathrm{id}\otimes \partial_{\mathbf{c}+\varepsilon_i,j}}&&
X_{\mathbf{b}}\otimes Y_{\mathbf{c}+\varepsilon_j+\varepsilon_i} \\
X_{\mathbf{b}}\otimes Y_{\mathbf{c}}
\ar[rr]_{\mathrm{id}\otimes \partial_{\mathbf{c},j}}\ar[u]^{\mathrm{id}\otimes d_{\mathbf{c},i}}&&
X_{\mathbf{b}}\otimes Y_{\mathbf{c}+\varepsilon_j}\ar[u]_{\mathrm{id}\otimes d_{\mathbf{c}+\varepsilon_j,i}}
}
\end{displaymath}
(depending on the parity of $\mathrm{ht}(\mathbf{b})$).
Clearly, both of them give the corresponding part of \eqref{axiom4}  for $(Z_{\bullet},\daleth_{\bullet},n+m)$.

If $i\leq n$ and $j>n$, then we obtain one of the following two situations:
\begin{displaymath}
\xymatrix{
X_{\mathbf{b}+\varepsilon_i}\otimes 
Y_{\mathbf{c}}\ar[rr]^{-\mathrm{id}\otimes \partial_{\mathbf{c},j}}&&
X_{\mathbf{b}+\varepsilon_i}\otimes Y_{\mathbf{c}+\varepsilon_j} \\
X_{\mathbf{b}}\otimes Y_{\mathbf{c}}
\ar[rr]_{\mathrm{id}\otimes \partial_{\mathbf{c},j}}\ar[u]^{d_{\mathbf{b},i}\otimes \mathrm{id}}&&
X_{\mathbf{b}}\otimes Y_{\mathbf{c}+\varepsilon_j}\ar[u]_{d_{\mathbf{b},i}\otimes \mathrm{id}}
}
\end{displaymath}
or
\begin{displaymath}
\xymatrix{
X_{\mathbf{b}+\varepsilon_i}\otimes 
Y_{\mathbf{c}}\ar[rr]^{\mathrm{id}\otimes \partial_{\mathbf{c},j}}&&
X_{\mathbf{b}+\varepsilon_i}\otimes Y_{\mathbf{c}+\varepsilon_j} \\
X_{\mathbf{b}}\otimes Y_{\mathbf{c}}
\ar[rr]_{-\mathrm{id}\otimes \partial_{\mathbf{c},j}}\ar[u]^{d_{\mathbf{b},i}\otimes \mathrm{id}}&&
X_{\mathbf{b}}\otimes Y_{\mathbf{c}+\varepsilon_j}\ar[u]_{d_{\mathbf{b},i}\otimes \mathrm{id}}
}
\end{displaymath}
(again, depending on the parity of $\mathrm{ht}(\mathbf{b})$).
Both of them give the corresponding part of \eqref{axiom4}  for $(Z_{\bullet},\daleth_{\bullet},n+m)$.

This, together with the observation that our tensor product of morphisms produces the usual tensor product
of morphisms of complexes after applying $\mathcal{F}$, implies that our tensor product is well-defined. 
All axioms of strict monoidal category now follow
directly from our construction as soon as we observe 
that the unit in  $\mathcal{P}(\mathcal{A})$ is the pyramid $(X_{\bullet},d_{\bullet},0)$, where 
$X_{\mathbf{0}}=\mathbbm{1}$, all other $X_{\mathbf{c}}=0$ and all $d_{\mathbf{c},i}=0$.
\end{proof}

\begin{corollary}\label{cor9}  
The equivalences of Theorem~\ref{thm2} and Corollary~\ref{cor3} are compatible with the
monoidal structure and are hence biequivalences.
\end{corollary}

\begin{proof}
This follows directly from the definitions. 
\end{proof}

\subsection{Pyramids and strict monoidal actions}\label{s2.6}
Let  $\mathcal{A}$ be as in the previous subsection and $\mathcal{C}$ an additive category
equipped with a strict monoidal action $\lozenge:\mathcal{A}\times \mathcal{C}\to \mathcal{C}$
by additive functors.

For a pyramid $(X_{\bullet},d_{\bullet},n)\in\mathcal{P}(\mathcal{A})$ and a pyramid
$(Y_{\bullet},\partial_{\bullet},m)\in\mathcal{P}(\mathcal{C})$ we define 
\begin{displaymath}
(X_{\bullet},d_{\bullet},n)\blacklozenge (Y_{\bullet},\partial_{\bullet},m)
\end{displaymath}
as the pyramid $(Z_{\bullet},\daleth_{\bullet},n+m)\in\mathcal{P}(\mathcal{C})$, where, for 
$\mathbf{a}\in\mathbb{I}$, we have
\begin{displaymath}
Z_{\mathbf{a}}:=X_{\pi_n(\mathbf{a})}\lozenge Y_{\sigma_n(\mathbf{a})},
\end{displaymath}
and, for $\mathbf{a}\in\mathbb{I}$ and $i\in \mathbb{N}$, we define
\begin{displaymath}
\daleth_{\mathbf{a},i}:=
\begin{cases}
\hspace{15mm}d_{\pi_n(\mathbf{a}),i}\lozenge\mathrm{id}, & \text{ if } i\leq n,\\
(-1)^{\mathrm{ht}(\pi_n(\mathbf{a}))}\,\,\,\,\,\,\mathrm{id}\lozenge\partial_{\sigma_n(\mathbf{a}),i}& \text{ otherwise}.
\end{cases}
\end{displaymath}

Let $\beta:(Y_{\bullet},\partial_{\bullet},m)\to(\tilde{Y}_{\bullet},\tilde{\partial}_{\bullet},\tilde{m})$
be a morphism of pyramids. We define the morphism $(X_{\bullet},d_{\bullet},n)\blacklozenge \beta$ as 
\begin{displaymath}
\gamma:
(X_{\bullet},d_{\bullet},n)\blacklozenge (Y_{\bullet},\partial_{\bullet},m)
\to({X}_{\bullet},{d}_{\bullet},{n})\blacklozenge (\tilde{Y}_{\bullet},\tilde{\partial}_{\bullet},\tilde{m})
\end{displaymath}
such that
\begin{displaymath}
(\gamma)_{\mathbf{a},\mathbf{b}}^{(k)}:=
\omega^{(\mathrm{ht}(\pi_n(\mathbf{a})))}_{\pi_n(\mathbf{a}),\pi_{\tilde{n}}(\mathbf{b})}\lozenge 
\beta^{(k-\mathrm{ht}(\pi_n(\mathbf{a}))}_{\sigma_n(\mathbf{a}),\sigma_{\tilde{n}}(\mathbf{b})},
\end{displaymath}
for any $k\in\mathbb{Z}$ and any $\mathbf{a},\mathbf{b}\in\mathbb{I}_k$
(recall the definition of $\omega^{(l)}$ from Section~\ref{s2.2}). 
This turns $(X_{\bullet},d_{\bullet},n)\blacklozenge {}_-$ into an additive endofunctor of 
$\mathcal{P}(\mathcal{C})$.

Finally, let $\alpha:(X_{\bullet},d_{\bullet},n)\to(\tilde{X}_{\bullet},\tilde{d}_{\bullet},\tilde{n})$
be a morphism of pyramids. We define 
\begin{displaymath}
\alpha\blacklozenge (Y_{\bullet},\partial_{\bullet},m):
(X_{\bullet},d_{\bullet},n)\blacklozenge (Y_{\bullet},\partial_{\bullet},m)\to
(\tilde{X}_{\bullet},\tilde{d}_{\bullet},\tilde{n})\blacklozenge (Y_{\bullet},\partial_{\bullet},m)
\end{displaymath}
as the morphism $\eta$ given by 
\begin{displaymath}
(\eta)_{\mathbf{a},\mathbf{b}}^{(k)}:=
\alpha^{(\mathrm{ht}(\pi_n(\mathbf{a})))}_{\pi_n(\mathbf{a}),\pi_{\tilde{n}}(\mathbf{b})}\lozenge 
\omega^{(k-\mathrm{ht}(\pi_n(\mathbf{a})))}_{\sigma_n(\mathbf{a}),\sigma_{\tilde{n}}(\mathbf{b})},
\end{displaymath}
for any $k\in\mathbb{Z}$ and any $\mathbf{a},\mathbf{b}\in\mathbb{I}_k$. 

\begin{proposition}\label{prop5}
The construct $\blacklozenge:\mathcal{P}(\mathcal{A})\times \mathcal{P}(\mathcal{C})\to \mathcal{P}(\mathcal{C}) $ 
is a strict monoidal action by additive functors. This action descends to a strict monoidal action
\begin{displaymath}
\blacklozenge:\mathcal{H}(\mathcal{A})\times \mathcal{H}(\mathcal{C})\to \mathcal{H}(\mathcal{C}). 
\end{displaymath}
\end{proposition}

\begin{proof}
Mutatis mutandis the proof of Proposition~\ref{prop4}.
\end{proof}

\section{Finitary $2$-categories and their $2$-representations}\label{s3}

In this section we recall basic facts from the classical $2$-representations theory
developed in \cite{MM1}-\cite{MM6}, see also \cite{Ma} for a survey and \cite{Ma0} for more details.

\subsection{Finitary $2$-categories}\label{s3.1}

Following \cite{MM1}, a {\em finitary} $2$-category $\cC$ over an algebraically closed field $\Bbbk$ 
is a $2$-category with finitely many objects in which each $\cC(\mathtt{i},\mathtt{j})$ is 
a small category equivalent 
to the category of projective modules for 
some finite  dimensional $\Bbbk$-algebra (which depends on both $\mathtt{i}$
and $\mathtt{j}$) and such that all compositions are (bi)additive and $\Bbbk$-linear and
all identity $1$-morphisms are indecomposable. 

In what follows, $\cC$ is always assumed to be a finitary $2$-category over $\Bbbk$.
All functors are assumed to be additive and $\Bbbk$-linear.

\subsection{$2$-representations}\label{s3.2}

A $2$-representation of $\cC$ is a $2$-functor to some fixed {\em target} $2$-category. 
All $2$-representations of $\cC$ form a $2$-category where $1$-morphisms are strong
natural transformations and $2$-morphisms are modifications, see \cite[Subsection~2.3]{MM3}. 
$2$-representations will be denoted by bold capital roman letters $\mathbf{M}$, $\mathbf{N}$ etc. 

Taking, as the target $2$-category, the $2$-category of finitary additive $\Bbbk$-linear categories,
we obtain the $2$-category $\cC\text{-}\mathrm{afmod}$ of {\em finitary}
$2$-representations of $\cC$. 
Taking, as the target $2$-category, the $2$-category of finitary abelian $\Bbbk$-linear categories,
we obtain the $2$-category $\cC\text{-}\mathrm{mod}$ of {\em abelian}
$2$-representations of $\cC$. 

There is a diagrammatic {\em abelianization} $2$-functor 
$\overline{\,\,\cdot\,\,}:\cC\text{-}\mathrm{afmod}\to \cC\text{-}\mathrm{mod}$,
see \cite[Subsection~3.1]{MM1}.

For each $\mathtt{i}\in \cC$, we have the {\em principal} $2$-representation
$\mathbf{P}_{\mathtt{i}}:=\cC(\mathtt{i},{}_-)$, for which we have the usual Yoneda lemma,
see \cite[Lemma~3]{MM3}.

\subsection{Simple transitive $2$-representations}\label{s3.3}

A finitary $2$-representation $\mathbf{M}$ of $\cC$ is called {\em transitive} provided that,
for any $\mathtt{i}$ and $\mathtt{j}$ and any indecomposable objects $X\in \mathbf{M}(\mathtt{i})$
and $Y\in \mathbf{M}(\mathtt{i})$, there is a $1$-morphism $\mathrm{F}$ of $\cC$ such that 
$Y$ is isomorphic to a direct summand of $\mathbf{M}(\mathrm{F})\, X$. 

A finitary $2$-representation $\mathbf{M}$ of $\cC$ is called {\em simple} provided that
it does not have any non-zero proper $\cC$-invariant ideals. We note that simplicity implies
transitivity, however, we will always speak about {\em simple transitive} $2$-representations. 
There is a weak Jordan-H{\"o}lder theory for finitary 
$2$-representations of $\cC$ developed in \cite{MM5}.

\subsection{Cells and cell $2$-representations}\label{s3.4}

For indecomposable $1$-morphisms $\mathrm{F}$ and $\mathrm{G}$ in $\cC$, we write
$\mathrm{F}\geq_L\mathrm{G}$ provided that $\mathrm{F}$ is isomorphic to a direct summand
of $\mathrm{H}\circ\mathrm{G}$, for some $1$-morphism $\mathrm{H}$. This defines the
{\em left} preorder $\geq_L$, equivalence classes of which are called {\em left cells}.
Similarly one defines the {\em right} preorder $\geq_R$ and {\em right cells}, and also
the {\em two-sided} preorder $\geq_J$ and {\em two-sided cells}.

A two-sided cell $\mathcal{J}$ is called {\em strongly regular} provided that no two of its left (or right)
cells are comparable with respect to the left (respectively right) order and the intersection
of any left and any right cell in $\mathcal{J}$ contains precisely one element.

Given a left cell $\mathcal{L}$, there is a unique $\mathtt{i}$ such that all $1$-morphisms
in $\mathcal{L}$ start at $\mathtt{i}$. The corresponding {\em cell $2$-representation}
$\mathbf{C}_{\mathcal{L}}$ is defined as the subquotient of $\mathbf{P}_{\mathtt{i}}$ obtained
by taking the unique simple transitive quotient of the subrepresentation of $\mathbf{P}_{\mathtt{i}}$
given by the additive closure of all $1$-morphisms $\mathrm{F}$ such that 
$\mathrm{F}\geq_L \mathcal{L}$. The $2$-representation $\mathbf{C}_{\mathcal{L}}$ is simple transitive.
We refer to \cite[Subsection~6.5]{MM2} for details.

If $\mathbf{M}$ is a simple transitive $2$-representation of $\cC$, then the set of two-sided cells
whose elements do not annihilate $\mathbf{M}$ contains a unique maximal element called the {\em apex}
of $\mathbf{M}$, see \cite[Subsection~3.2]{CM}.

\subsection{Bookkeeping tools}\label{s3.5}

Let $\mathbf{M}$ be a finitary $2$-representation of $\cC$. Then, to each $1$-morphism $\mathrm{F}$,
we can associate a matrix $[\mathrm{F}]$ with non-negative integer coefficients, 
whose rows and columns are indexed by isomorphism classes of indecomposable objects in 
\begin{displaymath}
\mathcal{M}:=\coprod_{\mathtt{i}} \mathbf{M}(\mathtt{i}),
\end{displaymath}
and the  $X\times Y$-entry gives the multiplicity of $X$ as a direct summand of $\mathbf{M}(\mathrm{F})\, Y$.

If we additionally know that $\overline{\mathbf{M}}(\mathrm{F})$ is exact, we also have the
matrix $\llbracket \mathrm{F}\rrbracket$ with non-negative integer coefficients, 
whose rows and columns are indexed by isomorphism classes of simple objects in $\overline{\mathcal{M}}$
and the  $X\times Y$-entry gives the composition multiplicity of $X$ in $\overline{\mathbf{M}}(\mathrm{F})\, Y$.

If $(\mathrm{F},\mathrm{G})$ is an adjoint pair of $1$-morphisms, then $\overline{\mathbf{M}}(\mathrm{G})$
is exact and $[\mathrm{F}]^t=\llbracket \mathrm{G}\rrbracket$, see \cite[Lemma~10]{MM5}.

\section{The $2$-category $\cD_A$ and its $2$-representations}\label{s7}

\subsection{Definition of  $\cD_A$}\label{s7.1}

Let $\Bbbk$ be an algebraically closed field and $A$ a connected, basic, finite dimensional associative (unital)
$\Bbbk$-algebra. Let $\mathcal{C}$ be a small category equivalent to $A\text{-}\mathrm{mod}$. As usual, we denote
by $*$ the $\Bbbk$-duality $\mathrm{Hom}_{\Bbbk}({}_-,\Bbbk)$. We define the
$2$-category $\cD_A=\cD_{A,\mathcal{C}}$ to have
\begin{itemize}
\item one object $\mathtt{i}$ (which we identify with $\mathcal{C}$);
\item as $1$-morphisms all endofunctors of $\mathcal{C}$ isomorphic to tensoring with 
$A$-$A$-bi\-mo\-du\-les in $\mathrm{add}\big(A\oplus (A\otimes_{\Bbbk}A)\oplus(A^{*}\otimes_{\Bbbk}A)\big)$;
\item as $2$-morphisms all natural transformations of functors.
\end{itemize}

We denote by $\mathrm{F}$ and $\mathrm{G}$ the functors given by tensoring with $A\otimes_{\Bbbk}A$
and $A^{*}\otimes_{\Bbbk}A$, respectively. We have the multiplication table for these
functors given by
\begin{equation}\label{eq7.3}
\begin{array}{c||c|c}
X\setminus Y&\mathrm{F}&\mathrm{G}\\
\hline\hline
\mathrm{F}&\mathrm{F}^{\oplus \dim(A)}&\mathrm{F}^{\oplus \dim(A)}\\
\mathrm{G}&\mathrm{G}^{\oplus \dim(A)}&\mathrm{G}^{\oplus \dim(A)} 
\end{array}
\end{equation}
and an adjoint pair $(\mathrm{F},\mathrm{G})$ of $1$-morphisms in $\cD_A$
(see e.g. \cite[Section~7.3]{MM1}).

We denote by $\mathcal{J}$ the unique two-sided cell for $\cD_A$ that does not contain the identity $1$-morphism.
It consists of the indecomposable constituents of $\mathrm{F}$ and $\mathrm{G}$.

\begin{proposition}\label{prop79}
The $2$-category $\cD_A$ is {\em $\mathcal{J}$-simple} in the sense that any non-zero 
$2$-ideal of $\cD_A$ contains the identity $2$-morphisms for all $1$-morphisms in $\mathcal{J}$.
\end{proposition}

\begin{proof}
Given a non-zero endomorphism of $\mathrm{F}\oplus\mathrm{G}$ corresponding to an $A$-$A$-bimodule
homomorphism
\begin{displaymath}
\varphi:(A\oplus A^*)\otimes_{\Bbbk}A \to (A\oplus A^*)\otimes_{\Bbbk}A,
\end{displaymath}
we can pre- and post-compose it with the identity on $A\otimes_{\Bbbk}A$ to obtain a non-zero non-radical
endomorphism of a direct sum of copies of $\mathrm{F}$. The claim follows.
\end{proof}

\subsection{Cell $2$-representations of $\cD_A$}\label{s7.5}

Let $\mathbf{N}$ denote the $2$-representation of $\cD_A$ given by the natural action of $\cD_A$
on the additive category generated by all projective and all injective objects in $\mathcal{C}$.

\begin{proposition}\label{prop75}
Let $\mathcal{L}$ be a left cell in $\mathcal{J}$. Then $\mathbf{C}_{\mathcal{L}}$ is equivalent
to $\mathbf{N}$.
\end{proposition}

\begin{proof}
Let $\epsilon_1,\epsilon_2,\dots,\epsilon_n$ be a complete list of pairwise orthogonal primitive
idempotents in $A$. Left cells of $\cD_A$ are indexed by $\{1,2,\dots,n\}$. Without loss of generality,
we may assume that $1$-morphisms in $\mathcal{L}$ correspond to the bimodules in the additive  closure of 
$(A\oplus A^*)\otimes_{\Bbbk} \epsilon_1A$.

We denote by $\mathbf{M}$ the defining $2$-representation of $\cD_A$ on $\mathcal{C}$.

Let $L_1$ be a simple object in $\mathcal{C}$ corresponding to $\epsilon_1$. Then we have a unique
morphism $\Phi:\mathbf{P}_{\mathtt{i}}\to \mathbf{M}$ sending $\mathbbm{1}_{\mathtt{i}}$ to $L_1$.
For $\mathrm{H}\in \mathcal{L}$, we have $\mathrm{H}\, L_1\in \mathbf{N}(\mathtt{i})$.
By the usual argument, see e.g. \cite[Proposition~22]{MM2}, $\Phi$ induces an equivalence from  
$\mathbf{C}_{\mathcal{L}}$ to $\mathbf{N}$.
\end{proof}

\subsection{Simple transitive $2$-representations of $\cD_A$}\label{s7.2}
Here we formulate our main result about the $2$-category $\cD_A$.

\begin{theorem}\label{thm71}
Each simple transitive $2$-representation of  $\cD_A$ is equivalent to a cell $2$-representation.
\end{theorem}

Before proving this theorem, we need some preparation.

\subsection{Some quasi-idempotent bimodules}\label{s7.3}

For a positive integer $k$, we denote by $\underline{k}$ the set $\{1,2,\dots,k\}$.

Let $B$ be a finite dimensional associative (unital)
$\Bbbk$-algebra. Let $M_1,M_2,\dots,M_k$ be a list of pairwise non-isomorphic indecomposable 
left $B$-modules. Let ${}_1N,{}_2N,\dots,{}_lN$ be a list of pairwise non-isomorphic indecomposable 
right $B$-modules. Let $H$ be a  $B$-$B$-bimodule of the form
\begin{equation}\label{eq7.1}
H=\bigoplus_{i=1}^k\bigoplus_{j=1}^l \big(M_i\otimes_{\Bbbk} {}_jN\big)^{\oplus h_{i,j}},
\end{equation}
where all $h_{i,j}\in \mathbb{Z}_{\geq 0}$.

\begin{proposition}\label{prop72}
Assume that the following conditions are satisfied.
\begin{enumerate}[$($a$)$]
\item\label{prop72.1} $H\otimes_B H\cong H^{\oplus d}$, for some $d\in \mathbb{N}$. 
\item\label{prop72.2} For each $i,j\in\{1,2,\dots,k\}$, the module $M_i$ is isomorphic to a direct summand of 
$H\otimes_B M_{j}$. 
\item\label{prop72.3} There is a decomposition $H\cong H_1\oplus H_2$ of $B$-$B$-bimodules such that we have 
$H_1\otimes_B H\cong H_1^{\oplus d'}$, for some $d'\in \mathbb{N}$.
\end{enumerate}
Then $H_1\cong M\otimes_{\Bbbk}N$, for some left $B$-module $M$ and some right $B$-module $N$.
\end{proposition}

\begin{proof}
Define the $k\times l$-matrix $\mathtt{H}=(h_{i,j})_{i\in\underline{k}}^{j\in\underline{l}}$
describing the multiplicities in \eqref{eq7.1}.
Define a $l\times k$-matrix $\mathtt{C}=(c_{j,i})^{i\in\underline{k}}_{j\in\underline{l}}$ 
via $c_{j,i}:=\dim({}_jN\otimes_B M_i)$. From \eqref{prop72.1}, we deduce that 
\begin{equation}\label{eq7.2}
\mathtt{H}\mathtt{C}\mathtt{H}=d\mathtt{H}
\end{equation}
and hence also $\mathtt{H}\mathtt{C}\mathtt{H}\mathtt{C}=d\mathtt{H}\mathtt{C}$.

The matrix $\mathtt{H}\mathtt{C}$ describes multiplicities of 
$M_i$ in $H\otimes_B M_{j}$, for $i,j\in\underline{k}$, and hence is positive by \eqref{prop72.2}. 
From $(\mathtt{H}\mathtt{C})^2=d \mathtt{H}\mathtt{C}$ and \cite[Proposition~4.1]{Zi2}
we see that $\mathrm{rank}(\mathtt{H}\mathtt{C})=1$. Now, \eqref{eq7.2} implies that 
$\mathrm{rank}(\mathtt{H})=1$.

Define the $k\times l$-matrix $\mathtt{H}_1=(h^{(1)}_{i,j})_{i\in\underline{k}}^{j\in\underline{l}}$
describing the multiplicities of $M_i\otimes_{\Bbbk} {}_jN$ in $H_1$. By \eqref{prop72.3}, we also have
$\mathtt{H}_1\mathtt{C}\mathtt{H}=d'\mathtt{H}_1$ and thus $\mathrm{rank}(\mathtt{H}_1)=1$.
Write $\mathtt{H}_1=vw^t$, for some $v\in\mathbb{Z}_{\geq 0}^{k}$ and $w\in\mathbb{Z}_{\geq 0}^{l}$.
Then, for 
\begin{displaymath}
M:=\bigoplus_{i=1}^k M_i^{\oplus v_i}\qquad\text{ and }\qquad 
N:=\bigoplus_{j=1}^l {}_jN^{\oplus w_j},
\end{displaymath}
we obtain $H_1\cong M\otimes_{\Bbbk}N$.
\end{proof}

\subsection{Proof of Theorem~\ref{thm71}}\label{s7.4}

\begin{proof}
Let $\mathbf{M}$ be a simple transitive $2$-representation of $\cD_A$.
If all $1$-morphisms in $\mathcal{J}$ annihilate $\mathbf{M}$, then $\mathbf{M}$ is a cell $2$-representation
by \cite[Theorem~18]{MM5}.

Assume now that $\mathbf{M}$ has apex $\mathcal{J}$. We denote by $B$ the
basic algebra such that $\mathbf{M}(\mathtt{i})$ is equivalent to $B\text{-}\mathrm{proj}$.
Let $e_1,e_2,\dots,e_m$ be a complete set of pairwise orthogonal primitive idempotents in $B$.
Then the Cartan matrix of $B$ is 
\begin{displaymath}
\mathtt{C}:=\big(\dim(e_iBe_j)\big)_{i\in\underline{m}}^{j\in\underline{m}}. 
\end{displaymath}

Let $X$ and $Y$ be the $B$-$B$-bimodules corresponding to the actions of $\mathbf{M}(\mathrm{F})$
and $\mathbf{M}(\mathrm{G})$, respectively. By \cite[Theorem~11(i)]{KMMZ} (which does not make use
of the fiatness assumption), both $X$ and $Y$ have the property that the $B$-modules 
$X\otimes_B L$ and $Y\otimes_B L$ are projective, for any $B$-module $L$. Therefore, by 
\cite[Theorem~1]{MMZ}, both $X$ and $Y$ are of the form
\begin{displaymath}
\bigoplus_{i=1}^m Be_i\otimes_{\Bbbk}{}_iN,
\end{displaymath}
for some right $B$-modules ${}_iN$. By \cite[Section~3]{MZ}, all ${}_iN$ are right projective.

By transitivity of $\mathbf{M}$ and \eqref{eq7.3}, we can apply Proposition~\ref{prop72} both
to the pair $H=X\oplus Y$ and $H_1=X$ and to the pair $H=X\oplus Y$ and $H_1=Y$. By 
Proposition~\ref{prop72}, we can write $X\cong M\otimes_{\Bbbk}N$, where $M$ is left $B$-projective
and $N$ is right $B$-projective. Similarly, we can write $Y\cong M'\otimes_{\Bbbk}N'$, where $M'$ 
is left $B$-projective and $N'$ is right $B$-projective.
Define $\mathbf{a},\mathbf{b},\mathbf{a}',\mathbf{b}'\in\mathbb{Z}_{\geq 0}^m$ by 
\begin{equation}\label{eq7.7}
M\cong \bigoplus_{i=1}^m Be_i^{\oplus a_i},\quad 
M'\cong \bigoplus_{i=1}^m Be_i^{\oplus a'_i},\quad 
N\cong \bigoplus_{i=1}^m e_iB^{\oplus b_i},\quad 
N'\cong \bigoplus_{i=1}^m e_iB^{\oplus b'_i}.
\end{equation}

Set $d:=\dim(A)$. On the one hand, $X\otimes_B Y\cong X^{\oplus d}$ and, on the other hand,
\begin{displaymath}
X\otimes_B Y\cong M\otimes_{\Bbbk}N\otimes_B M'\otimes_{\Bbbk}N'\cong
\big(M\otimes_{\Bbbk}N'\big)^{\oplus \mathbf{b}^t\mathtt{C}\mathbf{a}'}.
\end{displaymath}
Consequently
\begin{equation}\label{eq7.4}
d \mathbf{b}= \mathbf{b}^t\mathtt{C}\mathbf{a}'\mathbf{b}'.
\end{equation}
Similarly, on the one hand, $Y\otimes_B X\cong Y^{\oplus d}$ and, on the other hand,
\begin{displaymath}
Y\otimes_B X\cong M'\otimes_{\Bbbk}N'\otimes_B M\otimes_{\Bbbk}N\cong
\big(M'\otimes_{\Bbbk}N\big)^{\oplus (\mathbf{b}')^t\mathtt{C}\mathbf{a}}.
\end{displaymath}
Therefore
\begin{equation}\label{eq7.5}
d \mathbf{b}'= (\mathbf{b}')^t\mathtt{C}\mathbf{a}\mathbf{b}.
\end{equation}

Due to the adjunction $(\mathrm{F},\mathrm{G})$, we have 
$[\mathbf{M}(\mathrm{F})]^t=\llbracket\mathbf{M}(\mathrm{G})\rrbracket$. 
Using \eqref{eq7.7}, directly from the definitions we deduce that the
$i,j$-th component of the matrix $[\mathbf{M}(\mathrm{F})]$ is 
\begin{displaymath}
\sum_{r=1}^m a_i\dim(e_rBe_j)b_r = \sum_{r=1}^m a_i c_{r,j}b_r
\end{displaymath}
and therefore $[\mathbf{M}(\mathrm{F})]=\mathbf{a}\mathbf{b}^t\mathtt{C}$ which yields
 $[\mathbf{M}(\mathrm{F})]^t=\mathtt{C}^t\mathbf{b}\mathbf{a}^t$.
Similarly, we have  $\llbracket\mathbf{M}(\mathrm{G})\rrbracket=\mathtt{C}\mathbf{a}'(\mathbf{b}')^t$.
This implies 
\begin{equation}\label{eq7.6}
\mathtt{C}^t\mathbf{b}\mathbf{a}^t=\mathtt{C}\mathbf{a}'(\mathbf{b}')^t.
\end{equation}

By adjunction, we have 
\begin{displaymath}
\mathrm{End}_{B\text{-}B}(M\otimes_{\Bbbk} N)\cong
\mathrm{End}_{B\text{-}}(M) \otimes_{\Bbbk}\mathrm{End}_{\text{-}B}( N)
\end{displaymath}
and hence 
\begin{displaymath}
\dim(\mathrm{End}_{B\text{-}B}(M\otimes_{\Bbbk} N))=
\dim(\mathrm{End}_{B\text{-}}(M))\cdot \dim(\mathrm{End}_{\text{-}B}( N)). 
\end{displaymath}
From \eqref{eq7.7} we obtain
\begin{displaymath}
\dim(\mathrm{End}_{B\text{-}}(M))=\mathbf{a}^t\mathtt{C}\mathbf{a}\qquad\text{ and }\qquad
\dim(\mathrm{End}_{\text{-}B}(N))=\mathbf{b}^t\mathtt{C}\mathbf{b}.
\end{displaymath}
This allows us to compute
\begin{displaymath}
\begin{array}{rcl}
\dim(\mathrm{End}_{B\text{-}B}(M\otimes_{\Bbbk} N))&=&
(\mathbf{a}^t\mathtt{C}\mathbf{a})(\mathbf{b}^t\mathtt{C}\mathbf{b})\\
&=&(\mathbf{b}^t\mathtt{C}\mathbf{b})(\mathbf{a}^t\mathtt{C}\mathbf{a})\\
&=&(\mathbf{b}^t\mathtt{C}\mathbf{b})^t(\mathbf{a}^t\mathtt{C}\mathbf{a})\\
&=&(\mathbf{b}^t\mathtt{C}^t\mathbf{b})(\mathbf{a}^t\mathtt{C}\mathbf{a})\\
&=&\mathbf{b}^t\mathtt{C}\mathbf{a}'(\mathbf{b}')^t\mathtt{C}\mathbf{a},\\
\end{array}
\end{displaymath}
where in the third line we used that the transpose of a number is the same number and in the
last line we used \eqref{eq7.6}. We have no representation theoretic interpretation for
this crucial computation.  Then we have
\begin{displaymath}
(\mathbf{b}^t\mathtt{C}\mathbf{a}')((\mathbf{b}')^t\mathtt{C}\mathbf{a})\mathbf{b}\overset{\eqref{eq7.5}}{=}
d (\mathbf{b}^t\mathtt{C}\mathbf{a}')\mathbf{b}'\overset{\eqref{eq7.4}}{=}d^2\mathbf{b}.
\end{displaymath}
As $\mathbf{b}\neq 0$, it follows that 
$(\mathbf{b}^t\mathtt{C}\mathbf{a}')((\mathbf{b}')^t\mathtt{C}\mathbf{a})=d^2$
and hence 
\begin{displaymath}
\dim(\mathrm{End}_{B\text{-}B}(M\otimes_{\Bbbk} N))=\dim(A\otimes A^{\mathrm{op}}).
\end{displaymath}

Due to Proposition~\ref{prop79},
the $2$-functor $\mathbf{M}$ induces an embedding of $A\otimes A^{\mathrm{op}}$, which is the
endomorphism algebra of $\mathrm{F}$, into $\mathrm{End}_{B\text{-}B}(X)$. This embedding 
must be an isomorphism by the above dimension count. As $A$ is basic, the algebra 
$A\otimes A^{\mathrm{op}}$ is also basic and hence 
\begin{displaymath}
\mathrm{End}_{B\text{-}B}(M\otimes_{\Bbbk} N)\cong
\mathrm{End}_{B\text{-}}(M) \otimes_{\Bbbk}\mathrm{End}_{\text{-}B}( N)
\end{displaymath}
is basic as well. This means that both $M$ and $N$ are basic. Moreover, since primitive 
idempotents in $A\otimes A^{\mathrm{op}}$ and $\mathrm{End}_{B\text{-}B}(X)$ correspond and
$(\mathrm{F},\mathrm{G})$ is an adjoint pair, it follows
that all indecomposable $1$-morphisms in $\mathcal{J}$ correspond to  indecomposable projective
$B$-$B$-bimodules.

Let $\mathcal{L}$ be a left cell in $\mathcal{J}$ and $L$ a simple object in 
$\overline{\mathbf{M}}(\mathtt{i})$ which is not 
annihilated by $1$-morphisms in $\mathcal{L}$. Such $L$ exists since otherwise all $1$-morphisms
in $\mathcal{J}$ would act as zero. Let $\mathrm{K}_1,\mathrm{K}_2,\dots,\mathrm{K}_s$ be a
complete list of pairwise non-isomorphic $1$-morphisms in $\mathcal{L}$ and 
\begin{displaymath}
\mathrm{K}:=\mathrm{K}_1\oplus\mathrm{K}_2\oplus\cdots\oplus\mathrm{K}_s.
\end{displaymath}
Then $\mathbf{M}(\mathrm{K})\, L$ is a basic projective generator of $\overline{\mathbf{M}}(\mathtt{i})$.

We have the evaluation morphism
\begin{displaymath}
\Phi:\mathrm{End}_{\ccD_A}(\mathrm{K})\to  
\mathrm{End}_{\overline{\mathbf{M}}(\mathtt{i})}(\mathbf{M}(\mathrm{K})\, L). 
\end{displaymath}
By construction of cell $2$-representations, the kernel of the corresponding evaluation 
morphism $\Psi$ for the cell $2$-representation $\mathbf{C}_{\mathcal{L}}$ is the unique
maximal left $2$-ideal. Therefore the kernel of $\Phi$ is contained in the kernel of $\Psi$. 
On the other hand,
the image of $\Phi$ is a subalgebra of $B$. At the same time, the above computation shows that the
Cartan matrix of the algebra $Q$ underlying $\mathbf{C}_{\mathcal{L}}$ and that of $B$ coincide (as
both encode the structure constants of multiplication of $1$-morphisms in $\mathcal{J}$). Consequently,
the kernel of $\Phi$ must coincide with the kernel of $\Psi$ and $Q\cong B$. 

We have a unique homomorphism from $\mathbf{P}_{\mathtt{i}}$ to $\overline{\mathbf{M}}(\mathtt{i})$
sending $\mathbbm{1}_{\mathtt{i}}$ to $L$. By the above, this restricts to an equivalence between 
$\mathbf{C}_{\mathcal{L}}$ and $\mathbf{M}$. The proof is complete.
\end{proof}

\section{The $2$-category $\cC_A$ and its $2$-representations}\label{s8}

\subsection{Definition of  $\cC_A$}\label{s8.1}

Let $A$ and $\mathcal{C}$ be as in Subsection~\ref{s7.1}. Define the
$2$-category $\cC_A=\cC_{A,\mathcal{C}}$ to have
\begin{itemize}
\item one object $\mathtt{i}$ (which we identify with $\mathcal{C}$);
\item as $1$-morphisms all endofunctors of $\mathcal{C}$ isomorphic to tensoring with 
$A$-$A$-bi\-mo\-du\-les in $\mathrm{add}\big(A\oplus (A\otimes_{\Bbbk}A)\big)$;
\item as $2$-morphisms all natural transformations of functors.
\end{itemize}

Note that, by definition, $\cC_A$ is a $2$-subcategory of $\cD_A$.

We denote by $\mathrm{F}$ the functor corresponding to tensoring with $A\otimes_{\Bbbk}A$. 
We also denote by $\mathcal{J}'$ the two-sided cell for $\cC_A$ that does not contain the identity $1$-morphism.

\subsection{Cell $2$-representations of $\cC_A$}\label{s8.3}
Here we formulate a similar statement to Proposition~\ref{prop75}. Let $\mathbf{N}$ denote 
the $2$-representation of $\cC_A$ given by the natural action of $\cC_A$ on the 
additive category generated by all projective objects in $\mathcal{C}$.

\begin{proposition}\label{prop85}
Let $\mathcal{L}'$ be a left cell in $\mathcal{J}'$. Then $\mathbf{C}_{\mathcal{L}'}$ is equivalent
to $\mathbf{N}$.
\end{proposition}

\begin{proof}
Mutatis mutandis the proof of Proposition~\ref{prop75}.
\end{proof}

\subsection{Simple transitive $2$-representations of $\cC_A$}\label{s8.2}
Our main result is the following statement.

\begin{theorem}\label{thm81}
Each simple transitive $2$-representation of  $\cC_A$ is equivalent to a cell $2$-representation.
\end{theorem}

Special cases of this result were obtained in \cite[Theorem~15]{MM5}, \cite[Theorem~33]{MM6}, 
\cite[Theorem~1]{MZ}, \cite[Theorem~6]{MMZ}, \cite[Theorem~19]{MZ2}, \cite[Theorem~3.1]{Zi3}.

\begin{proof}
Let $\mathbf{M}$ be  a simple transitive $2$-representation of $\cC_A$. If all $1$-morphisms in 
$\mathcal{J}'$ annihilate $\mathbf{M}$, then $\mathbf{M}$ is a cell $2$-representation
by \cite[Theorem~18]{MM5}.
 
Assume now that the apex of $\mathbf{M}$ is $\mathcal{J}'$. Let $\mathcal{L}'$ be a left cell in
$\mathcal{J}'$. Consider the $2$-category $\mathcal{H}(\cC_A)$ and
its action on $\mathcal{H}(\mathbf{M}(\mathtt{i}))$.  Let
\begin{displaymath}
\dots\to Q_2\to Q_1\to Q_0\to 0 
\end{displaymath}
be a projective resolution of the $A$-$A$-bimodule $A^*\otimes_{\Bbbk}A$. Let 
$\mathrm{Q}$ be an object in $\mathcal{H}(\cC_A)$ which corresponds to this resolution under
the biequivalence between $\mathcal{H}(\cC_A)$ and 
$\mathcal{K}^-(\mathrm{add}\big(A\oplus(A\otimes_{\Bbbk} A)\big))$
in Corollary~\ref{cor9}.

Note that all $1$-morphisms in $\cD_A$ correspond to right $A$-projective $A$-$A$-bimodules 
and hence define exact endofunctors
of both ${\mathcal{C}}$ and its derived category $\mathcal{D}^-({\mathcal{C}})$.
Let $\cA$ be the $2$-subcategory of $\mathcal{H}(\cC_A)$ generated by $\mathrm{Q}$ and $\mathrm{F}$.
Then $\cA$ acts, after applying $\mathcal{F}$ from Theorem~\ref{thm2}, on 
$\mathcal{D}^-({\mathcal{C}})$ by functors which are isomorphic to 
the corresponding functors in $\cD_A$. As both actions are $2$-full and $2$-faithful, 
this induces a biequivalence between $\cD_A$ and $\cA$.

Denote by  $\mathbf{N}(\mathtt{i})$ the additive closure in $\mathcal{H}(\mathbf{M}(\mathtt{i}))$
of $\mathbf{M}(\mathtt{i})$ and $\mathrm{Q}\,\mathbf{M}(\mathtt{i})$. By construction, this is
a finitary additive $2$-representation of $\cA$. Note that the original $2$-representation 
$\mathbf{M}$ of $\cC_A$ is a $2$-subrepresentation of the restriction of $\mathbf{N}$ to $\cC_A$.
Let $\mathbf{N}'$ be the simple transitive $2$-subquotient of $\mathbf{N}$ containing this copy of 
$\mathbf{M}$.

By Theorem~\ref{thm71}, every simple transitive $2$-representation of $\cA$ is a cell $2$-representation.
In particular, $\mathbf{N}'$ must be equivalent to $\mathbf{C}_{\mathcal{L}}$, where 
$\mathcal{L}$ is a left cell in $\mathcal{J}$. The restriction of $\mathbf{C}_{\mathcal{L}}$ to $\cC_A$ 
contains a unique simple transitive subquotient with apex $\mathcal{J}'$ (as all simple objects in 
$\overline{\mathbf{C}}_{\mathcal{L}}$ which do not correspond to $1$-morphisms in $\mathcal{J}'$ are killed
by $1$-morphisms in $\mathcal{J}'$). By construction, the latter is equivalent to the cell 
$2$-representation of $\cC_A$ for the unique left cell $\mathcal{L}'$ contained of $\cC_A$ in $\mathcal{L}$. 
The claim follows.
\end{proof}

\begin{remark}\label{rem89}
{\em
Theorem~\ref{thm81} admits a straightforward generalization to the case when $A$ is not connected.
In this general case objects of $\cC_A$ are in a one-to-one correspondence with connected components
of $A$. 
}
\end{remark}

\subsection{A characterization of $\cC_A$}\label{s8.4}
In this subsection we give a characterization of $2$-categories of the form $\cC_A$ inside the class of
finitary $2$-categories.

\begin{theorem}\label{thm86}
Let $\cC$ be a finitary $2$-category. Assume that the following conditions are satisfied.
\begin{enumerate}[$($a$)$]
\item\label{thm86.1} $\cC$ has one object $\mathtt{i}$ and exactly two two-sided cells, 
namely, one consisting of the identity $1$-morphism and one other, called $\mathcal{J}$.
\item\label{thm86.2} $\mathcal{J}$ is strongly regular and has the same number of left cells as of right cells.
\item\label{thm86.3} $\cC$ is $\mathcal{J}$-simple.
\item\label{thm86.4} There is a left cell $\mathcal{L}$ in $\mathcal{J}$ such that the
corresponding cell $2$-representation $\mathbf{C}_{\mathcal{L}}$ is exact and $2$-full.
We denote by $A$ the algebra underlying $\mathbf{C}_{\mathcal{L}}$.
\item\label{thm86.5} The $2$-endomorphism algebra of $\mathbbm{1}_{\mathtt{i}}$ surjects
onto the center of $A$.
\end{enumerate}
Then $\cC$ is biequivalent to $\cC_A$.
\end{theorem}

\begin{proof}
We consider the cell $2$-representation $\mathbf{C}_{\mathcal{L}}$ of $\cC$. It is simple
transitive by construction. By \cite[Theorem~11(i)]{KMMZ}, all indecomposable $1$-morphisms in 
$\mathcal{J}$ act on $\overline{\mathbf{C}}_{\mathcal{L}}(\mathtt{i})$ as functors which send
any object to a projective object. By \cite[Theorem~1]{MMZ}, all indecomposable $1$-morphisms in 
$\mathcal{J}$ act on $\overline{\mathbf{C}}_{\mathcal{L}}(\mathtt{i})$ as functors isomorphic to
tensoring with $\Bbbk$-split bimodules. Thanks to the exactness part in \eqref{thm86.4}, 
all indecomposable $1$-morphisms in  $\mathcal{J}$ act on $\overline{\mathbf{C}}_{\mathcal{L}}(\mathtt{i})$ 
as projective functors, moreover, as indecomposable projective functors due to the 
$2$-fullness part of \eqref{thm86.4}. 

By $\mathcal{J}$-simplicity, the $2$-representation $\mathbf{C}_{\mathcal{L}}$ is $2$-faithful. 
Condition~\eqref{thm86.2} guarantees that all indecomposable projective functors on 
$\overline{\mathbf{C}}_{\mathcal{L}}(\mathtt{i})$ are in the essential image of 
$\overline{\mathbf{C}}_{\mathcal{L}}$.
Now \eqref{thm86.5} implies that the image of this $2$-representation is also $2$-full 
and hence induces a biequivalence with $\cC_A$.
\end{proof}

\vspace{5mm}

\noindent
Vo.~Ma.: Department of Mathematics, Uppsala University, Box. 480,
SE-75106, Uppsala, SWEDEN, email: {\tt mazor\symbol{64}math.uu.se}

\noindent
Va.~Mi.: School of Mathematics, University of East Anglia,
Norwich NR4 7TJ, UK,  {\tt v.miemietz\symbol{64}uea.ac.uk}

\noindent
X.~Z.: Department of Mathematics, Uppsala University, Box. 480,
SE-75106, Uppsala, SWEDEN, email: {\tt xiaoting.zhang\symbol{64}math.uu.se}

\end{document}